\documentclass{amsart}

\usepackage{amssymb}
\usepackage{amscd}
\usepackage[all,cmtip]{xy}
\usepackage{hyperref}
\usepackage{graphicx}

\newcommand{\comment}[1]{}

\newcommand{\Z}{\mathbf{Z}}

\newcommand{\F}{\mathbf{F}}

\newcommand{\res}{\mathrm{res}}

\DeclareMathOperator{\Res}{Res}

\DeclareMathOperator{\Gal}{Gal}

\theoremstyle{plain}
\newtheorem{theorem}{Theorem}[section]

\newtheorem{corollary}[theorem]{Corollary}

\newtheorem{lemma}[theorem]{Lemma}

\newtheorem{proposition}[theorem]{Proposition}

\newtheorem{assumption}[theorem]{Assumption}

\theoremstyle{definition}

\newtheorem{remark}[theorem]{Remark}

\newtheorem{example}[theorem]{Example}

\newtheorem{problem}{Problem}

\begin{document}

\title{Notes on summation polynomials}
\author[]{Michiel Kosters (TL@NTU, mkosters@ntu.edu.sg), Sze Ling Yeo (I2R, slyeo@i2r.a-star.edu.sg)}
\address{}
\email{}
\urladdr{}
\date{\today}

\maketitle
%\tableofcontents

\begin{abstract}
In these short notes, we will show the following. Let $\F_q$ be a finite field and let $E/\F_q$ be an elliptic curve. Let $S_r$ be the $r$th summation/Semaev polynomial for $E$.

\begin{itemize}
\item Under an assumption, we show that it is NP-complete to check if $S_r$ for some large $r$ evaluates to zero on some input. Unconditionally, we prove a similar result for summation polynomials over singular curves. This suggests limitations in the usage of summation polynomials in for example algorithms to solve the elliptic curve discrete logarithm problem. 
\item Assume that $q$ is a power of $2$. We show that the Weil descent to $\F_2$ of $S_3$ for ordinary curves in general has first fall degree $2$, which is much lower than expected. The reason is the existence of a group morphism to $\F_2$ which gives a linear polynomial after Weil descent. We want to raise awareness of its existence and raise doubt on certain Gr\"obner basis heuristics which claim that the first fall degree is close to the degree of regularity. Furthermore, this morphism can be used to speed up the relation generation to solve the elliptic curve discrete logarithm problem.
\end{itemize}

\end{abstract}

\section{Introduction}

Let $\F_q$ be a finite field of cardinality $q$ and let $E/\F_q$ be an elliptic curve. Let $P \in E(\F_q)$ be a rational point and let $Q \in \langle P \rangle$. The \emph{elliptic curve discrete logarithm problem} (ECDLP for short) is to find an integer $m$ such that $mP=Q$. The apparent hardness of this problem is of great importance in cryptography as it forms the backbone of the security of various elliptic curve-based cryptographic primitives such as in the Diffie-Hellman key exchange protocol.

Various attacks on the ECDLP exist. For certain types of curves, fast algorithms exist (Silverman, \cite[XI.6]{SI}), but for generic elliptic curves, no sub-exponential algorithm is known.

Motivated by the sub-exponential index calculus attack for the discrete logarithm problem in finite fields, attempts were made to mimic such attacks for elliptic curves. In general, such attacks on ECDLP focus on the generation of relations. When enough relations have been obtained, one can solve the discrete logarithm using linear algebra. Let $E/\F_q$ be an elliptic curve given in Weierstrass model and let $x: E(\F_q) \setminus \{0\} \to \F_q$ be the $x$-coordinate map. For every integer $r \in \Z_{\geq 3}$, one can define the $r$th summation/Semaev polynomial $S_r \in \F_q[X_0,\ldots,X_{r-1}]$ for $E$. This polynomial has the following property. Let $x_0,\ldots,x_{r-1} \in \overline{\F_q}$. Then one has $S_r(x_0,\ldots,x_{r-1})=0$ if and only if there are $P_i \in E(\overline{\F_q})$ with $x(P_i)=x_i$ and $P_0+\ldots+P_{r-1}=0$ (Proposition \ref{1}). Such summation polynomials have been used to obtain the required relations between points on an elliptic curve (see for example Diem \cite{DIE}). In articles such as \cite{DIE}, \cite{FAU}, \cite{FAU2} and \cite{GAL} people suggest and try to work with summation polynomials where $r$ is large and they often handle them using symmetric properties of these polynomials. In \cite{FAU2}, the corresponding authors, for example, compute the $8$th summation polynomials.

The goal of this article is twofold. First, we want to show the limitations of the summation polynomial.  Under the assumption that one can construct elliptic curves over finite fields together with a point of large order, we show that it is  NP-complete to check that the $r$-th summation polynomial evaluates to zero on some input for large $r$ (Theorem \ref{102}i). Furthermore, we define summation polynomials for singular curves (see Section \ref{39}).  We unconditionally prove a similar result for summation polynomials coming from singular curves (Theorem \ref{102}ii). We prove these statements by reducing 3-SAT to the subsets sum problem and then to the problem concerning summation polynomials. We remark that these results do not imply that ECDLP is NP-complete (Remark \ref{144}).

Second, let $\F=\F_{2^n}$ and let $E/\F$ be an elliptic curve given by $Y^2+a_1 XY+a_3 Y = X^3+a_2 X^2+a_4 X+a_6$  such that $a_1 \neq 0$.  Petit and Quisquater in \cite{PET} suggest that the \emph{degree of regularity}, an important parameter in the complexity analysis of Gr\"obner basis calculations, of specific Weil descent systems coming from the ECDLP for $E$ is close to the \emph{first fall degree} of such a system. This assumption allows the authors to heuristically obtain sub-exponential algorithms for ECDLP. Their heuristic assumption is largely motivated by a similar and widely-believed conjecture concerning a Weil descent system arising from a univariate polynomial. Besides, they performed some experiments with small parameters ($n \le 17$) for Weil descent systems from the third summation polynomial to show that the degree of regularity in these cases is close to the bound on the first fall degree they give.    In this paper, we explicitly show that the first fall degree in this case in general is $2$ (Corollary \ref{w8} and Remark \ref{w11}).  The reason for this unexpectedly low first fall degree is the existence of a surjective morphism which factors through taking $x$-coordinates:
\begin{align*}
E(\F) \to&  \F_2 \\
P \mapsto& \mathrm{Tr}_{\F/\F_2} \left( \frac{x(P)+a_2}{a_1^2} \right).
\end{align*}

On the other hand, we performed further experiments to investigate the first fall degree assumption for $n$ upto $40$. Our results indicate that contrary to the assumption, the degree of regularity seems to grow as $n$ increases. This raises doubts to the heuristic assumption, and consequently, the heuristic sub-exponential complexity estimate for the ECDLP in \cite{PET}. 

Next, we point out  that even though the trace morphism is known, as far as we are aware, it has not been utilized in ECDLP computations. Indeed, this morphism can be used to speed up Gr\"obner bases calculations to solve ECDLP (Remark \ref{109}). 

Finally, we will comment on the recent preprints by Semaev \cite{SEM} and Karabina \cite{KARA} based on the results of this article. 

The remainder of this paper is organized as follows. In Section 2, we review the definition and properties of summation polynomials. Section 3 is dedicated to our first main result, namely, the NP-completeness of the evaluation of a summation polynomial on a given input. In Section 4, we describe the trace morphism which leads us to determine the first fall degree of the Weil descent system arising from the third summation polynomial for an elliptic curve over a finite field of characteristic $2$. Finally, we wrap up the paper with some experimental results on the degree of regularity of such systems and we discuss the results of Semaev \cite{SEM} and Karabina \cite{KARA}.     

\section{Summation polynomials} \label{39}

In this section, we will define summation polynomials for a general elliptic curve in Weierstrass form.

Let $F$ be a field and let $A=(a_1,a_2,a_3,a_4,a_6) \in F^5$. Set 
\begin{align*}
b_2 =& a_1^2+4 a_2, \\
b_4 =& a_1 a_3+2 a_4, \\
b_6 =&  a_3^2+4a_6, \\
b_8=& a_1^2 a_6-a_1a_3a_4+a_2 a_3^2+4a_2 a_6-a_4^2.
\end{align*}
We define
\begin{align*}
S_{A,2}=X_0-X_1 \in F[X_0,X_1].
\end{align*}
We define the third summation polynomial to be the polynomial $S_{A,3} \in F[X_0,X_1,X_2]$ of degree $4$ by:
\begin{align*}
S_{A,3} =& (X_0^2 X_1^2+X_0^2 X_2^2+X_1^2 X_2^2)-2 \cdot (X_0^2 X_1 X_2 +X_0 X_1^2 X_2 + X_0 X_1 X_2^2)   \\
& - b_2 \cdot (X_0 X_1 X_2) - b_4 \cdot (X_0 X_1+X_0X_2 + X_1X_2) -b_6  \cdot (X_0+X_1+X_2)-b_8.
\end{align*}
We will quite often write  $S_A$ instead of $S_{A,3}$.
For $r \in \Z_{>3},$ we recursively define the $r$th summation polynomial by
\begin{align*}
S_{A,r}=\Res_{X}\left(S_{A,r-1}(X_0,\ldots,X_{r-3},X) , S_{A,3}(X_{r-2},X_{r-1},X) \right) \in F[X_0,\ldots,X_{r-1}],
\end{align*}

where $\res_{X}$ denotes the resultant with respect to $X$.

We have the following proposition.

\begin{proposition} \label{1}
Let $F$ be a field and let $E/F$ be an elliptic curve given by $Y^2+a_1 XY+a_3 Y = X^3+a_2 X^2+a_4 X+a_6$. Let $r \in \Z_{\geq 2}$ and let $x_0,\ldots,x_{r-1} \in \bar{F}$. Then there are $P_0, \ldots, P_{r-1} \in E(\overline{F}) \setminus \{0\}$ with $x(P_i)=x_i$ ($i=0,\ldots,r-1$) such that $P_0+\ldots+P_{r-1}=0$ if and only if $S_{(a_1,a_2,a_3,a_4,a_6),r}(x_0,\ldots,x_{r-1})=0$. 
\end{proposition}
\begin{proof}
From the definition of the resultant, one directly sees that it is enough to prove the case $r=2, 3$. See \cite{DIE}, especially Lemma 3.4.
\end{proof}

The next proposition describes two degenerate cases of the summation polynomial.

\begin{proposition} \label{2}
Let $F$ be a field. Let $r \in \Z_{\geq 2}$. One has the following.

\begin{enumerate}
\item Let $x_0,\ldots,x_{r-1} \in \overline{F}^*\setminus \{1\}$. Then there are $n_i \in \{-1,1\}$ ($i=0,\ldots,r-1$) such that $x_0^{n_0} \cdots x_{r-1}^{n_{r-1}}=1$ if and only if 
\begin{align*}
S_{(1,0,0,0,0),r} ( x_0/(x_0-1)^2, \ldots, x_{r-1}/(x_{r-1}-1)^2 ) =0.
\end{align*}
\item Let $x_0,\ldots,x_{r-1} \in \overline{F} \setminus \{0\}$. Then there are $n_i \in \{-1,1\}$ ($i=0,\ldots,r-1$) such that $n_0x_0+\ldots+n_{r-1}x_{r-1}=0$ if and only if 
\begin{align*}
S_{(0,0,0,0,0),r} (1/x_0^2, \ldots, 1/x_{r-1}^2 ) =0.
\end{align*}
\end{enumerate}
\end{proposition}
\begin{proof}
By properties of the resultant, it is enough to prove the proposition for $r=2, 3$. The case $r=2$ is an easy calculation. 

i. Assume $r=3$.
There are $n_i$ as above if and only if 
\begin{align*}
0=& (x_0 x_1 x_2-1) \cdot (x_0^{-1} x_1 x_2-1) \cdot (x_0 x_1^{-1} x_2-1) \cdot (x_0 x_1 x_2^{-1} -1) \\
=& (-x_0 x_1 x_2)^{-1} \cdot (x_1 x_2-x_0) (x_0x_2 -x_1) (-x_0x_1 + x_2) (x_0x_1x_2 - 1).
\end{align*}
A calculation shows that
\begin{align*}
& S_{(1,0,0,0,0)} \left( \frac{x_0}{(x_0-1)^2}, \frac{x_1}{(x_1-1)^2}, \frac{x_2}{(x_2-1)^2} \right)  = \\
& \left( (x_0-1)(x_1-1)(x_2 - 1) \right)^{-4} \cdot (x_1x_2 -x_0) (x_0x_2 -x_1) (-x_0x_1 + x_2) (x_0 x_1 x_2 - 1).
\end{align*}
Hence the result follows.

ii. The proof is similar to the proof of i, because one has
\begin{align*}
 & S_{(0,0,0,0,0)} \left( \frac{1}{x_0^2}, \frac{1}{x_1^2}, \frac{1}{x_2^2} \right) = \\
& \left( x_0 x_1 x_2 \right)^{-4} \cdot (-x_0 + x_1 - x_2 ) (-x_0 + x_1 + x_2) (x_0 + x_1 - x_2) (x_0 + x_1 + x_2).
\end{align*}
\end{proof}

\begin{remark}
The resemblance between Proposition \ref{1} and Proposition \ref{2} is no coincidence.  

Let $F$ be a field. Consider the nodal curve $E$ given by $y^2+xy-x^3=0$ (Weierstrass model $(1,0,0,0,0)$). Let $E_{ns}(F)$ be the non-singular locus of $E$ over $F$. We have an isomorphism:
\begin{align*}
F^* \to& E_{ns}(F) \\
1 \mapsto& 0 \\
t \mapsto& (t/(t-1)^2,t/(t-1)^3).
\end{align*}
The inverse is given by $0 \mapsto 1$ and $(x,y) \mapsto 1+x/y$.
See \cite[Chapter III, Proposition 2.5]{SI}. 

For ii consider the cuspidal curve $E$ given by $y^2=x^3$. One has $E(F) \cong F$ in this case.

Finally, there is also the case of a nodal elliptic curve where the tangent line at the node is not rational. In this case, one has $E(F) \cong \ker(\mathrm{Norm}_{F'/F})$ where $F'/F$ is a quadratic extension of $F$ (\cite[Theorem 2.31]{WAS}). One should be able to use similar summation polynomials in this case.
\end{remark}

\section{NP-completeness of summation polynomials} \label{580}

We will now study NP-completeness properties of summation polynomials. Most results in this section were already known. See for example \begin{verbatim}
https://ellipticnews.wordpress.com/2011/08/04/hard-problems-of-
algebraic-geometry-codes-by-qi-cheng/
\end{verbatim} 

We would like to warn the reader that the above result does not imply that summation polynomials are not helpful for solvingt the elliptic curve discrete logarithm problem (Remark \ref{144}).

We begin with the following problem.

\begin{problem}[Subset sum problem]
Given a finite abelian group $G$, a subset $S \subseteq G$ and $g \in G$, determine if there is a subset $T \subseteq S$ such that $\sum_{t \in T} t=g$.
\end{problem}

We start with a known result, although the proof for $m=3$ below might be new.

\begin{proposition} \label{104}
The subset sum problem for the following sets of groups is NP-complete:
\begin{enumerate}
\item $\{\Z\}$, $\{\Z/n\Z,\ n \in \Z_{\geq 1}\}$;
\item $\{(\Z/m\Z)^n,\ n \in \Z_{\geq 1}\}$ for $m \geq 3$.
\end{enumerate}
\end{proposition}
\begin{proof}
All problems are obviously in NP.

i. This result for $\{\Z\}$ was shown in \cite{KAR}. A proof can also be given as in ii, using some $m$-adic representation of integers. The result for $\{\Z/n\Z,\ n \in \Z_{\geq 1}\}$ follows directly from the result for $\{\Z\}$.

ii. Fix $m$. We first look at another problem. We look for an $r \in \Z_{\geq 1}$, $k \in \Z_{\geq 0}$ and vectors $c_1, c_2, c_3 \in (\Z/m\Z)^r$, $d_1,\ldots,d_k \in (\Z/m\Z)^r$, $t \in (\Z/m\Z)^r$ with the following properties:
\begin{enumerate}
\item given a non-empty subset of $C \subseteq \{c_1,c_2,c_3\}$, there is a subset of $D \subseteq \{d_1,\ldots,d_k\}$ such that $\sum_{c \in C} c + \sum_{d \in D} d=t$;
\item no subset of $\{d_1,\ldots,d_k\}$ sums to $t$.
\end{enumerate}

Suppose we have found a solution to the above problem. We will show how to reduce an instance of $3$-SAT to the subset sum problem in $(\Z/m\Z)^n$ for some small $n$. Assume that the $3$-SAT instance has variables $x_1,\ldots,x_s$, with negations $\overline{x_1},\ldots,\overline{x_s}$ and that there are $w$ clauses. An example of such a clause would be $x_1 \vee \overline{x_2} \vee x_4$. We will now translate this to a subset sum problem in $R=(\Z/m\Z)^{s} \times \left( (\Z/m\Z)^r \right)^w$. We represent an element $R$ as $(a_1,\ldots,a_s, b_1, \ldots,b_t)= \sum_{i=1}^s a_i e_i + \sum_{j=1}^w b_j e_j'$ where $a_i \in \Z/m\Z$ and $b_j \in (\Z/m\Z)^r$ and the $e_i$ and $e_j'$ are the standard basis vectors. Set $c_0=0 \in (\Z/m\Z)^r$ and set $\pi(x_i)=\pi(\overline{x_i})=i$. 

Let $x$ be a variable or its negation. We define $v_x$ as follows. For $j=1,\ldots,t$ define a function $r_x(j) \in \{0,1,2,3\}$ as follows. If $x$ appears in clause $j$ for the first time at position $r \in \{1,2,3\}$ set $r_x(j)=r$. If $x$ does not appear, set $r_x(j)=0$. We set
\begin{align*}
v_x = e_{\pi(x)} + \sum_{j=1}^w c_{r_x(j)} e_j' \in R.
\end{align*}
Furthermore, for $j=1,\ldots,w$ and $i=1,\ldots,k$ set
\begin{align*}
h_{j,i} = d_i e_j' \in R.
\end{align*}
Finally, set
\begin{align*}
w = \sum_{i=1}^s e_i + \sum_{j=1}^w  t e_j'.
\end{align*}
One easily obtains: the $3$-SAT instance has a solution if and only if there is a subset of $\{v_{x_1},\ldots,v_{x_s}, v_{\overline{x_1}},\ldots,v_{\overline{x_m}}\} \cup \{ h_{j,i}: j=1,\ldots,t,\ i=1,\ldots,k\}$ summing to $w$.

It remains to show that we can find the required $r$, $k$, $c_i$, $d_i$ and $t$. Assume first that $m >3$. One can take $r=1$, $k=2$, $c_1=c_2=c_3=1$, $d_1=d_2=1$ and $t=3$.
For $m=3$, it is harder to solve the problem. Set $r=3$, $k=5$ and set
\begin{align*}
c_1=(2,1,2),\ c_2=(2,2,2),\ c_3=(2,0,1),
\end{align*}
\begin{align*}
t=(2,0,1)
\end{align*}
and 
\begin{align*}
d_1=t-c_1,\ d_2=t-c_2,\ d_3=t-c_1-c_2,\ d_4=t-c_1-c_3,\ d_5=t-c_2-c_3.
\end{align*}
One easily verifies that this works.

\end{proof}

\begin{remark}
The proof of Proposition \ref{104}ii fails for $m=2$. The subset sum problem over $(\F_2)^n$ is easy: it is just linear algebra.
\end{remark}

Consider the following assumption.

\begin{assumption} \label{100}
Given a positive integer $n$, one can construct a finite field $\F_q$ of cardinality $q$ and an elliptic curve $E/\F_q$ together with a point $P \in E(\F_q)$ with $\mathrm{ord}(P) \geq n$ in polynomial time in $\log(n)$.
\end{assumption}

\begin{remark}
In a non-deterministic way, one can randomly find a curve $E/\F_q$ with $\# E(\F_q)$ prime and a non trivial point on this curve. One can do this since there are a lot of primes by the prime number theorem and one can count points on curves efficiently by Schoof's algorithm. See \cite{BROK} for more advanced methods on constructing elliptic curves with a prescribed number of points.
\end{remark}

We will now prove NP-completeness of summation polynomials.

\begin{theorem} \label{102}
The following hold.
\begin{enumerate}
\item Assume that Assumption \ref{100} holds. The following problem is NP-complete: given $\F_q$ be a finite field of cardinality $q$, $E/\F_q$ an elliptic curve in Weierstrass form with coefficients $A$, $r \in \Z_{\geq 3}$ an integer and $x_i \in \F_q$ ($i=0,\ldots,r-1$), determine if $S_{A,r}(x_0,\ldots,x_{r-1})$ is zero. 
\item 
Let $p \geq 3$ be a fixed prime. The following problem is NP-complete: given positive integers $n, r$, a finite field $\F_{p^n}$ of cardinality $p^n$, and $a_0,\ldots,a_{r-1} \in \F_{p^n} \setminus \{0\}$, determine if $S_{(0,0,0,0,0),r}(a_0,\ldots,a_{r-1})$ is zero.
\end{enumerate}
\end{theorem}
\begin{proof}
i. First of all, notice that this problem is in NP: a witness consists of $n_i \in \{\pm 1\}$ and $P_i \in E(\overline{\F_q})$ with $x(P_i)=x_i$ such that $\sum_i n_i P_i=0$ (Proposition \ref{1}). 

Suppose we are given a subset sum problem for the group $\Z$. Say we need to find $\epsilon_i \in \{0,1\}$ such that $\sum_{i=1}^{m} \epsilon_i v_i = w$. Note that $\sum_{i=1}^{m} \epsilon_i v_i = w$ if and only if $\sum_{i=1}^{m} 2\epsilon_i v_i = 2w$.
Hence the system is equivalent to solving for $n_i \in \{\pm 1\}$ the equation
\begin{align*}
 \sum_{i=1}^{m} n_i v_i = 2w- \sum_{i=1}^{m} v_i=w'.
\end{align*}
Use Assumption \ref{100} to construct a finite field $\F_q$ and a curve $E/\F_q$ with a point $P$ of order at least $1+\sum_i 2v_i$ with Weierstrass coefficients $A$. Then the above holds if and only if 
\begin{align*}
 \sum_{i=1}^{m} n_i v_iP=w'P.
\end{align*}
Assume that $w' \neq 0$. 
The later by Proposition \ref{1} is equivalent to 
\begin{align*}
S_{A,m+1}(x(v_1P),\ldots,x(v_{m}P),x(w'P))=0. 
\end{align*}
If $w'=0$, one can use $S_{A,m}$. Hence the result follows from Proposition \ref{104}i.

ii. The proof is very similar to the proof of i. The problem is in NP by Proposition \ref{2}. Suppose we are given an instance of a subset sum problem over $(\F_p)^m$. After multiplying by $2$ we reduce to the problem of checking if there are $n_i \in \{\pm 1\}$ with say 
\begin{align*}
 \sum_{i=1}^{m} n_i v_i =w'
\end{align*}
with $v_i \neq 0$.

Find an irreducible polynomial $f$ over $\F_p$ of degree $m$ and construct a field $\F_{p^{m}}=\F_p[X]/(f)$ (one can do this since $p$ is fixed, see \cite{SHO2}). Identify $\F_{p}^{m}$ with $\F_{p^{m}}$ using a linear isomorphism. Assume that $w' \neq 0$. Then the last problem is equivalent to checking if $S_{(0,0,0,0,0),m+1}(1/v_1^2,\ldots,1/v_{m}^2,1/w'^2)$ evaluates to zero by Proposition \ref{2}. If $w'=0$, then one can use a lower summation polynomial. Use Proposition \ref{104}ii to finish the proof.
\end{proof}

\begin{remark}
The proof of Theorem \ref{102}ii fails for $p=2$. For $p=2$, one has $S_{(0,0,0,0,0),r}(a_0,\ldots,a_{r-1})$ is $0$ if and only if $a_0+\ldots+a_{r-1}=0$ (Proposition \ref{2}). Hence the problem is very easy.
\end{remark}

\begin{remark} \label{144}
Theorem \ref{102} shows that it is NP-complete to check if summation polynomials evaluate to zero. However, it does not suggest that ECDLP itself is a hard problem. In fact, ECDLP for curves with for example $p$ points can be solved quickly (\cite[Chapter XI, Proposition 6.5]{SI}), but the above proof shows that it is still NP-complete to evaluate the corresponding summation polynomials.
\end{remark}

\section{Weil descent and first fall degrees}

In this section, we will study Weil descent systems coming from summation polynomials over a finite field of characteristic $2$. In particular, we study the system coming from the third summation polynomial from an ordinary elliptic curve. Let us first define the procedure of Weil descent.

\subsection{Weil descent}
Let $p$ be a prime and $n, r \in \Z_{\geq 1}$. Let $\F_{p^n}$ be a field of cardinality $p^n$. Consider
\begin{align*}
R_1=\F_{p^n}[X_1,\ldots,X_r]/(X_i^{p^n}-X_i: i=1,\ldots,r)
\end{align*}
and
\begin{align*}
R_2= \F_p[X_{ij}, i=1,\ldots r,\ j=1,\ldots,n]/(X_{ij}^p-X_{ij},\ i=1,\ldots,r,\ j=1,\ldots,n).
\end{align*}
Finally, set 
\begin{align*}
R_3= \F_{p^n}[X_{ij}, i=1,\ldots r,\ j=1,\ldots,n]/(X_{ij}^p-X_{ij},\ i=1,\ldots,r,\ j=1,\ldots,n).
\end{align*}

One has $R_1 \cong \left(\F_{p^n} \right)^{\F_{p^n}^r}$ and $R_2 \cong \left(\F_p \right)^{\F_p^{nr}}$ as rings, by evaluating the $X_i$ and $X_{ij}$ at points of $\F_{p^n}$ respectively $\F_p$. There is a bijection between the ideals of $R_1$ and the powerset of $\F_{p^n}^r$, and similarly, a bijection between the ideals of $R_2$ and the powerset of $\F_p^{nr}$.

Let $\alpha_1,\ldots,\alpha_n$ be a basis of $\F_{p^n}$ over $\F_p$. This gives us an isomorphism over $\F_p$ between $\F_{p^n}$ and $(\F_p)^n$, and hence one between $\F_{p^n}^r$ and $\left(\F_p \right)^{nr}$. This gives a bijection between the set of ideals of $R_1$ and $R_2$. We call this Weil descent.

The correspondence in practice is given as follows.
Let $f \in R_1$. Set $X_i= \sum_{j=1}^n X_{ij} \alpha_j$. Write 
\begin{align*}
f(\sum_{j=1}^n X_{1j} \alpha_j, \ldots,\sum_{j=1}^n X_{rj} \alpha_j )=\sum_{i=1}^n [f]_i \alpha_i \in R_3
\end{align*}
with $[f]_i \in R_2$. An ideal $I \subseteq R_1$ is mapped to $([f]_i: f \in I, i=1,\ldots,n) \subseteq R_2$. 

With Weil descent one can solve systems over $\F_{p^n}$, by solving systems over $\F_p$.

We make $R_1$ into a $\Z[G]$ module, where $G=\Gal(\F_{p^n}/\F_p)=\langle \mathrm{Frob} \rangle$ by setting
\begin{align*}
\mathrm{Frob}(f)=f^p.
\end{align*} 
For $f \in R_1$ we set
\begin{align*}
\mathrm{Tr}_{\F_{p^n}/\F_p}(f)  = \sum_{g \in G} g(f) = \sum_{i=0}^{n-1} f^{p^i} \in R_1.
\end{align*}
This defines a group morphism which extends the trace map $\mathrm{Tr}_{\F_{p^n}/\F_p}: \F_{p^n} \to \F_p$.  Furthermore, we have $\mathrm{Tr}_{\F_{p^n}/\F_p}(f) \in (f)$. Finally, for $a \in \F_{p^n}$ we have
\begin{eqnarray*}
\mathrm{Tr}_{\F_{p^n}/\F_p}(f(a)) = \mathrm{Tr}_{\F_{p^n}/\F_p}(f)(a).
\end{eqnarray*}

One has the following lemma.

\begin{lemma} \label{w15}
Let $f \in R_1$. Write $1=\sum_{i=1}^n c_i \alpha_i$ with $c_i \in \F_p$. Let $c \in \F_{p^n}$. Then for $i=1,\ldots,n$ one has
\begin{align*}
[\mathrm{Tr}_{\F_{p^n}/\F_p}(cf)]_i = c_i \sum_{j=1}^n \mathrm{Tr}_{\F_{p^n}/\F_p}(c\alpha_j) [f]_j \in R_2.
\end{align*}
If $c_j \neq 0$, then one has $([\mathrm{Tr}_{\F_{p^n}/\F_p}(cf)]_1,\ldots,[\mathrm{Tr}_{\F_{p^n}/\F_p}(cf)]_r)=([\mathrm{Tr}_{\F_{p^n}/\F_p}(cf)]_j)$.
\end{lemma}
\begin{proof}
One has 
\begin{align*}
cf \left( \sum_{j=1}^n X_{1j} \alpha_j, \ldots,\sum_{j=1}^n X_{rj} \alpha_j \right) = \sum_{i=1}^n (c \alpha_i) [f]_i \in R_3.
\end{align*}
Note that $[f]_i^p=[f]_i \in R_2$. Taking traces gives us the following identity in $R_3$:
\begin{align*}
\mathrm{Tr}_{\F_{p^n}/\F_p}(cf) (\sum_{j=1}^n X_{1j} \alpha_j, \ldots,\sum_{j=1}^n X_{rj} \alpha_j  ) &= \sum_{i=1}^n \mathrm{Tr}_{\F_{p^n}/\F_p} (c\alpha_i) [f]_i \\
&= \sum_{j=1}^n \left( c_j \sum_{i=1}^n \mathrm{Tr}_{\F_{p^n}/\F_p} (c\alpha_i) [f]_i \right) \alpha_j.
\end{align*}
This gives the first result. The second result follows directly.
\end{proof}

\subsection{Weil descent in characteristic $2$}

We are interested in the Weil descent of systems coming from summation polynomials.

\begin{proposition} \label{w1}
Let $\F=\F_{2^n}$ be a finite field of cardinality $2^n$. Let $E/\F$ be an elliptic curve given by $Y^2+a_1 XY+a_3 Y = X^3+a_2 X^2+a_4 X+a_6$. Assume that $E$ is ordinary ($a_1 \neq 0$). Then we have a surjective group morphism
\begin{align*}
E(\F) \to&  \F_2 \\
0 \mapsto& 0 \\
P \mapsto& \mathrm{Tr}_{\F/\F_2} \left( \frac{x(P)+a_2}{a_1^2} \right)
\end{align*}
with kernel $2E(\F)$.
\end{proposition}

\begin{proof}
We will only prove that the map is a group morphism.
Let $P_i=(x_i,y_i) \in E(\F)$, $i=1,2.$
If one of the $P_i$ is 0 or their sum is 0, the additivity of the map is clear.

Otherwise the line $L$ through $P_1$ and $P_2$ (the tangent line to $E$ if $P_1=P_2$) has an equation of the form
$$L: y=\lambda x+\nu.$$
Suppose $P_3=(x_3,y_3)$ is the third point of $L\cap E$. Then we have $P_1+P_2+P_3=0$ and the equation of $E$ gives us $x_1+x_2+x_3=\lambda^2+a_1\lambda+a_2$.

This gives 
\begin{align*}
\frac{x_1+a_2}{a_1^2}+\frac{x_2+a_2}{a_1^2}+\frac{x_3+a_2}{a_1^2}=\left(\frac{\lambda}{a_1}\right)^2+\frac{\lambda}{a_1}.
\end{align*}
Notice that $\mathrm{Tr}_{\F_{2^n}/\F_2}((\frac{\lambda}{a_1})^2)=\mathrm{Tr}_{\F_{2^n}/\F_2}(\frac{\lambda}{a_1})$.
Thus we have
\begin{align*}
\mathrm{Tr}_{\F_{2^n}/\F_2}\left(\frac{x_1+a_2}{a_1^2}\right)+\mathrm{Tr}_{\F_{2^n}/\F_2}\left(\frac{x_2+a_2}{a_1^2}\right)+
\mathrm{Tr}_{\F_{2^n}/\F_2}\left(\frac{x_3+a_2}{a_1^2}\right)=0.
\end{align*}
Therefore the additivity of the map follows.
See \cite[Chapter 7,  Proposition 5.4]{KO6} for a proof of the surjectivity of the map.
\end{proof}

From the above proposition, one sees for example that if $P \in E(\F_{2^n})$, then one has $P \in 2 E(\F_{2^{2n}})$.

\begin{corollary} \label{w2}
Let $\F=\F_{2^n}$ be a finite field of cardinality $2^n$. Let $E/\F$ be an elliptic curve given by $Y^2+a_1 XY+a_3 Y = X^3+a_2 X^2+a_4 X+a_6$. Assume that $E$ is ordinary. Let $P_1,\ldots,P_m \in E(\F)$. Assume that $\pm P_1 \pm \ldots \pm P_m=0$. Then one has
\begin{align*}
0=\sum_{i=1}^m \mathrm{Tr}_{\F/\F_2} \left( \frac{x(P_i)+a_2}{a_1^2} \right).
\end{align*}
\end{corollary}
\begin{proof}
The proof follows directly from Proposition \ref{w1}.
\end{proof}

\begin{lemma} \label{w3}
Let $\F=\F_{2^n}$ be a finite field of cardinality $2^n$. Let $r \in \Z_{\geq 2}$. Let $E/\F$ be an elliptic curve. Suppose that $Q \in E(\F)$, $P_i  \in E(\overline{\F}) \setminus E(\F)$ with $x(P_i) \in \F$ ($i=1,\ldots,r$) such that $Q=P_1+\ldots+P_r$. Then one has $2Q=0$.
\end{lemma}
\begin{proof}
Let $\F' \subseteq \overline{\F}$ be the unique quadratic extension of $\F$ in $\overline{\F}$. Then one has $P_i \in E(\F')$. Let $G=\langle \sigma \rangle=\Gal(\F'/\F)$ of order $2$. Note that $G$ acts on $E(\F')$ by $\sigma((x:y:z))=(\sigma(x):\sigma(y):\sigma(z))$. As $x(P_i) \in \F$, we conclude that $x(\sigma(P_i))=\sigma(x(P_i))=x(P_i)$. Hence we obtain $\sigma(P_i)=\pm P_i$. As $P_i \not \in E(\F)$, we find $\sigma(P_i)=-P_i$. Then we have:
\begin{align*}
2Q&=Q+\sigma(Q) =P_1+\ldots + P_r + \sigma(P_1) + \ldots + \sigma(P_r)\\
&=P_1-P_1+ \ldots+P_r-P_r=0
\end{align*}
as required.
\end{proof}

\begin{remark}
Note that one always has $E(\overline{\F})[2] \subseteq E(\F)$, by Galois invariance.
\end{remark}

\begin{example}
Consider the elliptic curve defined by $y^2 + xy = x^3 + 1$ over $\F_2$. One has $E(\F_2) \cong \Z/4\Z$ and $E(\F_{2^2}) \cong \Z/8\Z$. There are no points in $E(\F_{2^2}) \setminus E(\F_2)$ with $x$-coordinate in $\F_2$. Hence sometimes there are no $Q$ as in Lemma \ref{w3}. In most cases one can find such $Q$ with a decomposition.
\end{example}

\begin{proposition} \label{w6}
Let $\F=\F_{2^n}$ be a finite field of cardinality $2^n$. Let $E/\F$ be an elliptic curve given by $Y^2+a_1 XY+a_3 Y = X^3+a_2 X^2+a_4 X+a_6$. Assume that $E$ is ordinary. Let $S(X_1,X_2,X_3)$ be the $3$rd summation polynomial for $E$. Let $P \in E(\F) \setminus E(\F)[2]$. Consider the ideal $I=(S(X_1,X_{2},x(P))) \subseteq \F[X_1,X_{2}]/(X_1^{2^n}-X_1, X_2^{2^n}-X_2)=R$. Then one has
\begin{align*}
g= \mathrm{Tr}_{\F/\F_2}\left( \frac{X_1+X_2 +  x(P)+a_2}{a_1^2} \right) \in I.
\end{align*}
\end{proposition}
\begin{proof}
Suppose $(x_1,x_{2}) \in Z(I)$ where $x_i \in \overline{\F}$. Then one has $x_i \in \F$. By definition there are $P_i \in E(\overline{\F})$ ($i=1,2$) with $x(P_i)=x_i$ such that $P_1+P_{2}+P=0$. Note that $P_1 \not \in E(\F)$ iff $P_2 \not \in E(\F)$. By Lemma \ref{w3}, it follows that $P_1, P_2 \in E(\F)$.
Corollary \ref{w2} gives
\begin{align*}
\mathrm{Tr}_{\F/\F_2} \left( \frac{x(P_1)+x(P_2)+x(P)+a_2}{a_1^2} \right)=0.
\end{align*}
Hence we obtain $g(x_1,x_2)=0$. Hence we find $Z(I) \subseteq Z(g)$. Since $I$ is a radical ideal, by the Nullstellensatz we conclude $g \in I$.
\end{proof}

\begin{remark} \label{109}
Proposition \ref{w6} does not directly generalize to any $S_m$ with $m>3$. Indeed, we cannot always apply Lemma \ref{w3}. Consider the $m$-th summation polynomial, with $m$ even. Let $Q \in E(\overline{\F}) \setminus E(\F)$ with $x(Q) \in \F$ (such points exist if $n \geq 3$). Then one has $P=P+Q-Q+\ldots+Q-Q$. This shows that other decompositions exist. Similarly, for $m$ odd one can construct such examples.

Hence the relation of Proposition \ref{w6} is not always present in our ideal. But in applications, such as relation generation for the elliptic curve discrete logarithm problem, one can just add the equation from the start (Proposition \ref{w1}). Another option is to only look for relations in the kernel of the map $E(F) \to \F_2$ in Proposition \ref{w1}.
\end{remark}

A more explicit version of Proposition \ref{w6} is the following.

\begin{proposition} \label{w7}
Let $\F=\F_{2^n}$ be a finite field of cardinality $2^n$. Let $E/\F$ be an elliptic curve given by $Y^2+a_1 XY+a_3 Y = X^3+a_2 X^2+a_4 X+a_6$. Assume that $E$ is ordinary. Let $S(X_1,X_2,X_3)$ be the $3$rd summation polynomial for $E$. Let $P \in E(\F) \setminus E(\F)[2]$. Set $T=S(X_1,X_2,x(P)) \in \F[X_1,X_2]/(X_1^{2^n}-X_1,X_2^{2^n}-X_2)=R$. Set $b=a_1(a_1 x(P)+a_3) \in \F^*$. Then in $R$ one has
\begin{align*}
\mathrm{Tr}_{\F/\F_2}\left( T/b^2 \right)=\mathrm{Tr}_{\F/\F_2}\left( \frac{X_1+X_2 +  x(P)+a_2}{a_1^2} \right).
\end{align*}
\end{proposition}
\begin{proof}
Set $x=x(P)$. Note that $b \neq 0$, because $a_1 \neq 0$ (ordinary curve) and $a_1x+a_3 \neq 0$ ($P$ is not $2$-torsion). One has
\begin{align*}
T/b^2& =\left( \frac{1}{b} X_1 X_2 \right) ^2+ \frac{1}{b} X_1X_2+\left( \frac{x}{b} (X_1+X_2) \right)^2 \\
& +  \frac{a_3}{ba_1}(X_1+X_2)+\frac{b_6x+b_8}{b^2}.
\end{align*}
Note that $\mathrm{Tr}_{\F/\F_2}(\left( \frac{1}{b} X_1 X_2 \right) ^2+ \frac{1}{b} X_1X_2)=0$. Furthermore, one has
\begin{align*}
\frac{x}{b}+\frac{a_3}{ba_1}= \frac{a_1x+a_3}{ba_1}= \frac{1}{a_1^2}.
\end{align*}
This gives $\mathrm{Tr}_{\F/\F_2} \left( \left( \frac{x}{b} (X_1+X_2) \right)^2 +  \frac{a_3}{ba_1}(X_1+X_2)  \right) = \mathrm{Tr}_{\F/\F_2}( \frac{X_1+X_2}{a_1^2})$.
Now it remains to show that
\begin{align} \label{e1}
\mathrm{Tr}_{\F/\F_2}\left(\frac{b_6x+b_8}{b^2}\right)=\mathrm{Tr}_{\F_{2^n}/\F_2}\left(\frac{x+a_2}{a_1^2} \right).
\end{align}
If both expressions are different, then from Proposition \ref{w6} it follows that $1 \in I=(S(X_1,X_{2},x(P)))$. By assumption, $2P \neq 0$. We have a relation $P-2P+P=0$. From Proposition \ref{1} we obtain $(x(2P),x(P)) \subseteq Z(I)$, contradicting that $1 \in I$.
\end{proof}

\begin{remark}
One can prove Equation \ref{e1} as follows in a more computational way. Since $P$ is a point of the curve, one has 
\begin{align*}
\mathrm{Tr}_{\F/\F_2} \left( \frac{x^3+a_2x^2+a_4x+a_6}{(a_1x+a_3)^2} \right)=0.
\end{align*}
One has
\begin{align*}
\frac{b_6x+b_8}{b^2} = \frac{x+a_2}{a_1^2} +  \frac{x^3+a_2x^2+a_4x+a_6}{(a_1x+a_3)^2} + \frac{a_4}{b}+\left(\frac{a_4}{b} \right)^2.
\end{align*}
Note that the trace of $a_4/b$ and $a_4^2/b^2$ are the same. Taking traces gives us the required identity.
\end{remark}

After Weil descent we finally obtain the main result of this section.

\begin{corollary} \label{w8}
Let $\F=\F_{2^n}$ be a finite field of cardinality $2^n$. Let $E/\F$ be an elliptic curve given by $Y^2+a_1 XY+a_3 Y = X^3+a_2 X^2+a_4 X+a_6$. Assume that $E$ is ordinary. Let $S(X_1,X_2,X_3)$ be the $3$rd summation polynomial for $E$. Let $P \in E(\F) \setminus E(\F)[2]$ and set $T=S(X_1,X_2,x(P)) \in \F[X_1,X_2]/(X_1^{2^n}-X_1,X_2^{2^n}-X_2)$. Set $b=a_1(a_1 x(P)+a_3) \in \F^*$. Let $\alpha_1,\ldots,\alpha_n$ be a basis of $\F$ over $\F_2$. Then one has in $R_2$
\begin{align*}
\sum_{j} \mathrm{Tr}_{\F/\F_2}\left( \frac{\alpha_j}{b^2} \right) [T]_j =\mathrm{Tr}_{\F/\F_2} \left( \frac{x(P)+a_2}{a_1^2} \right) + \sum_{j=1}^n \mathrm{Tr}_{\F/\F_2}\left( \frac{\alpha_j}{a_1^2} \right) \cdot \left( X_{1j}+X_{2j} \right) .
\end{align*}
\end{corollary}
\begin{proof}
 Write $1= \sum_{i=1}^n c_i \alpha_i$ with $c_i \in \F_2$. Let $i$ be such that $c_i \neq 0$. Set $h=\frac{X_1+X_2+x(P)+a_2}{a_1^2}$.
By Proposition \ref{w7} one has
\begin{align*}
[\mathrm{Tr}_{\F/\F_2}(T/b^2)]_i  = [\mathrm{Tr}_{\F/\F_2}(h)]_i.
\end{align*}
By Lemma \ref{w15} the left hand side is equal to
\begin{align*}
c_i \sum_{j} \mathrm{Tr}_{\F/\F_2}\left( \frac{\alpha_j}{b^2} \right) [S]_j.
\end{align*}
Set $d=\mathrm{Tr}_{\F/\F_2} \left( \frac{x(P)+a_2}{a_1^2} \right) $. By Lemma \ref{w15} the right hand side is equal to
\begin{align*}
c_i  \left( d+ \sum_{j} \mathrm{Tr}_{\F/\F_2}\left( \frac{\alpha_j}{a_1^2} \right) [X_1+X_2]_j \right) = c_i  \left(d + \sum_{j} \mathrm{Tr}_{\F/\F_2}\left( \frac{\alpha_j}{a_1^2} \right) \left(  X_{1j}+X_{2j} \right) \right).
\end{align*}
This gives the result.
\end{proof}

\begin{remark} \label{w11}
In \cite{PET} the following is written: ``We have $D_{reg} \geq D_{firstfall}$. Experimental and theoretical evidences have shown in various contexts that the two definitions often lead to very close numbers.'' The above Corollary shows that this is not the case for $3$rd summation polynomials. Let us explain.

The right hand side of the equation of Corollary \ref{w8} always has degree $1$, whereas the $[T]_j$ on the left hand side usually has degree $2$. Hence it is likely that summing up certain polynomials of degree $2$ gives a polynomial of degree $1$ (in practice, this almost always happens). This, by definition of the first fall degree, shows that the first fall degree of the system given by  $S(X_1,X_2,x(P))$ after Weil descent is usually equal to $2$. Hence the first fall degree will be much smaller than the upper bound $5$ in Proposition 1 from \cite{PET} for a system consisting of a $3$rd summation polynomial. Furthermore, computations seem to suggest that the degree of regularity increases when $n$ increases. Here, the degree of regularity refers to the largest degree reached during Gr\"{o}bner basis computations using algorithms such as $F_4$ or $F_5$.

The following table records the degree of regularity  for the Weil descent system comprising the bivariate polynomial $S(X_1, X_2, x(P))$ for a random elliptic curve $E$ and a random point $P$ on $E$ over $\F_{2^n}$. Following the formulation in \cite{PET}, we include linear constraints on $X_1$ and $X_2$ to restrict their values to be in a random subspace of $\F_{2^n}$ of dimension $\lceil n/2\rceil$. We performed our computations using the ``GroebnerBasis()" function in the Magma computer Algebra System and the degree of regularity is read off from the Magma output as the largest step degree in which new polynomials were obtained in the step or the subsequent steps after setting the verbose to a nonzero value. Note that in all our computations, the first fall degree is $2$ as expected.

Here, the last fourth column in the table records the step at which the degree of regularity is first reached.

\begin{center}

\begin{tabular}{|c|c|c|c|c|} \hline
$n$ & First fall degree & Degree of regularity & Step & Memory\\ \hline
$12$&$2$&$3$&$3$ &$11.1$ MB\\ \hline
$16$&$2$&$3$&$3$&$11.1$ MB\\ \hline
$17$&$2$&$4$&$5$& $15.3$ MB\\ \hline
$20$&$2$&$4$&$5$& $30.2$ MB\\ \hline
$30$&$2$&$4$&$5$& $324.8$ MB\\ \hline
$40$&$2$&$\geq 5$&$ \geq 9$& $> 38$ GB\\ \hline
\end{tabular}

\end{center}

As the computations require more than $38$ GB for $n = 40,$ (they procedure was stopped because of lack of ram) we are not able to carry out more experiments for larger values of $n$. However, the behaviour of the step degrees and the drastic increase in memory suggest that the degree of regularity is $5$ or more when $n \geq 40$. This in turn indicates that the degree of regularity  follows an increasing pattern as $n$ increases. This raises doubt to the evidence of Assumption 2 from the article \cite{PET}: the gap between the degree of regularity and the first fall degree might be dependent on $n$.
\end{remark}

\section{Some recent developments}

In light of the recent articles written by Semaev \cite{SEM} and Karabina \cite{KARA}, we would like to point out our reservations of their claims as a result of the consequences of this article. We will focus on the first article, since the second article is quite similar.

In both articles, the authors claim that the degree of regularity of their systems is constant (in \cite{SEM} it is constantly $4$). We carried out the experiments of \cite{SEM} with $n=45$, $m=2$ and $t=2$. The only difference with the experiments in Remark \ref{w11} is that the sub vector space constraining the variables is not random. We observed that the degree of regularity increased to $5$. About $126 GB$ of RAM was used for this experiment and we completely finished the computation. For the case $n=40$, $m=2$ and  $t=2$, the degree of regularity stayed at $4$. Apparently, the choice of the specific vector space is a good one. We still believe that the degree of regularity will increase in all cases, and hence that the first fall degree assumption is very questionable.
Furthermore, $n=25$, $m=3$ and $t=3$ also seem to give degree of regularity $5$. We were not able to finish this computation after using $111$ GB of RAM.

In an updated version of the article of \cite{SEM}, another assumption about the growth of the degree of regularity has been added (according to this assumption, the degree of regularity grows slowly with certain parameters, just slow enough to obtain nice conclusions). When studying similar systems some time ago (including the splitting trick), we decided not to put up such a conjecture because we realised it would be very hard to verify (or falsify) this claim computationally. Furthermore, we could not come up with any reasoning which would support such heuristics.

One of the problems with the first fall degree assumption is that it does not `see' the number of variables. Let us give an extreme example in which we `prove' P=NP using the first fall degree assumption. In fact, the reason we wrote the first part of this article is an example related to this one.
In Section \ref{580}, we proved that it is NP-complete to check if a summation polynomial evaluates to zero or not. Let $S_m$ be the $m$-th summation polynomial for say an elliptic curve $E$ over a finite field $k$ of characteristic $2$. Suppose we want to determine if $S_m$ evaluated at $a_1,\ldots,a_m \in k$ is $0$. This is equivalent in checking if the folllowing ideal in $k[X_1,\ldots,X_{m-3}]$ contains $1$ by the splitting trick:
\begin{eqnarray*}
S_3(a_1,a_2,X_1) \\
S_3(a_3,X_1,X_2) \\
\ldots \\
S_3(a_{m-2},X_{m-4},X_{m-3}) \\
S_3(a_{m-1},a_m,X_{m-3}).
\end{eqnarray*}
We now perform Weil descent on the system to $\F_2$. The first and last equation are linear while the remaining ones are of the form $S_3(x,y,a)$. Consequently, the first fall degree of this system is $2$ (Corollary \ref{w8}). Under the first fall degree assumption, which says that the degree of regularity of such systems is bounded, we obtain a polynomial time algorithm (polynomial in the input) to solve the above problem. This seems highly unlikely.

It is certainly a very interesting question to derive good heuristical or theoretical bounds on the degree of regularity for systems as in \cite{SEM}: if indeed the degree of regularity is small, this splitting trick would give a good algorithm. Unfortunately, it is not even clear to us how to make a good heuristical bound, let alone a theoretical bound.

\section{Acknowledgements}

The authors would like to thank Ming-Deh A. Huang, Bagus Santoso, Chaoping Xing and Yun Yang for their help in preparing this manuscript. The authors are grateful to Steven Galbraith for his comments. Finally, we would like to thank the Caramel team from Nancy (France) for allowing us to use their computers to do experiments.

\bibliographystyle{acm}

\end{document}